\documentclass[reqno,12pt]{amsart}
\usepackage{mathrsfs}
\usepackage{amscd}
\usepackage{amsmath}
\usepackage{latexsym}
\usepackage{amsfonts}
\usepackage{amssymb}
\usepackage{amsthm}
\usepackage{graphicx}

\def\R{\mathbb{R}}

\numberwithin{equation}{section}

\newtheorem{thm}{Theorem}[section]
\newtheorem{lem}{Lemma}[section]

\newtheorem{remark}{Remark}[section]

\makeatletter
\newcommand{\Extend}[5]{\ext@arrow0099{\arrowfill@#1#2#3}{#4}{#5}}
\makeatother

\begin{document}

\setcounter{page}{1}

\title[Global rough solution for semilinear heat equation]{Global rough solution for $L^2$-critical semilinear heat equation in the negative Sobolev space}

\author{Avy Soffer}
\address{Rutgers University\\
Department of Mathematics\\
110 Frelinghuysen Rd.\\
Piscataway, NJ, 08854, USA\\}
\address{Department of Mathematics\\
Hubei Key Laboratory of Mathematical Science\\
Central China Normal University\\
Wuhan 430079, China.\\}
\email{soffer@math.rutgers.edu}
\thanks{}

\author{Yifei Wu}
\address{Center for Applied Mathematics\\
Tianjin University\\
Tianjin 300072, China}
\email{yerfmath@gmail.com}
\thanks{}

\author{Xiaohua Yao}
\address{Department of Mathematics\\
Hubei Key Laboratory of Mathematical Science\\
Central China Normal University\\
Wuhan 430079, China.\\}
\email{yaoxiaohua@mail.ccnu.edu.cn}


\date{}

\maketitle

\begin{abstract}\noindent
In this paper, we consider the Cauchy global problem for the $L^2$-critical semilinear heat equations
$
\partial_t h=\Delta h\pm |h|^{\frac4d}h,
$
with $h(0,x)=h_0$, where  $h$ is an unknown real function defined on $ \R^+\times\R^d$.  In most of the studies on this
subject, the initial data $h_0$ belongs to Lebesgue spaces $L^p(\R^d)$ for some $p\ge 2$ or
to subcritical Sobolev space $H^{s}(\R^d)$ with $s>0$. We here prove that  there exists some positive constant $\varepsilon_0$ depending on $d$, such that the Cauchy problem is locally and globally well-posed for any initial data $h_0$ which is radial, supported away from origin
and in the negative Sobolev space $\dot H^{-\varepsilon_0}(\R^d)$ including $L^p(\R^d)$ with certain $p<2$ as subspace. Furthermore, unconditional uniqueness, and $L^2$-estimate both as time $t\to0$ and $t\to +\infty$ were considered.

\end{abstract}


 \baselineskip=20pt

\section{Introduction}
Consider the initial value problem for a semilinear heat equation:
\begin{equation}\label{heat'}
   \left\{ \aligned
    &\partial_t h=\Delta h\pm |h|^{\gamma-1}h, \\
    &h(0,x)=h_0(x), \quad x\in \R^d,
   \endaligned
  \right.
 \end{equation}
where  $h(t,x)$ is an unknown real function defined on $ \R^+\times\R^d$, $d\ge 2$, $\gamma>1$. The positive sign ``+" in nonlinear term of \eqref{heat'} denotes focusing source,   and the negative sign ``-" denotes the defocusing one. The Cauchy problem \eqref{heat'} has been extensively studied in Lebesgue space $L^p(\R^d)$ by many peoples, see e.g. \cite{Bre-Caz, Bre-Fre, Bre-Pel, Ga-Va, Giga, IKNW, Ma-Mer, Mer, Miao-Zhang, Monn, Ni-Sacks, Ribaud, Ruf-Terrano, Tan, Weiss1, Weiss2} and so on. The equation
 enjoys an interesting property of scaling invariance
$$
h_\lambda(t,x):=\lambda^{2/(\gamma-1)}h(\lambda^2t, \lambda x), \ h_\lambda(0,x):= \lambda^{2/(\gamma-1)}h_0(\lambda x),  \ \lambda>0,
$$
that is, if $h(t,x)$ is the solution of heat equation \eqref{heat'}, then $h_\lambda(t,x)$ also does with the scaling data $\lambda^{2/\gamma}h_0(\lambda x)$.
An important fact is that Lebsgue space $L^{p_c}(\R^d)$ with $p_c=\frac{d(\gamma-1)}{2}$ is the only one invariant under the same scaling transform: $$h_0(x)\mapsto\lambda^{2/(\gamma-1)}h_0(\lambda x).$$

If we consider the initial data $h_0\in L^p(\R^d)$, then the scaling index
$$p_c=\frac{d(\gamma-1)}{2}$$
plays a critical role on the local/global well-posedness of \eqref{heat'}. Roughly speaking, one can divide the dynamics of \eqref{heat'} into the following three different regimes:  (A) {\it the subcritial case $p>p_c $},  (B) {\it the critical case $p=p_c$}, (C) {\it the supercritical case $p<p_c$}.  Specifically, In cases (A) and (B), i.e.  $p\ge p_c$, when  $p>\gamma$, Weissler in \cite{Weiss1} proved the local existence and uniqueness of solution $h\in C([0, T); L^q(\R^d))\cap L^\infty_{loc}((0,T]; L^\infty(\R^d)) $. Later, Brezis and Cazenave \cite{Bre-Caz} proved the unconditional uniquessness of Weissler's solution.  In {\it double critical case} $p=p_c=\gamma$ ( i.e. $p=\gamma=\frac{d}{d-2}$ ), the local conditional wellposedness of the problem \eqref{heat'} was due to Weissler in \cite{Weiss2}, but the unconditional uniqueness fails, see Ni-Sacks \cite{Ni-Sacks}, Terraneo \cite{Terran}. In the supercritical case $(C)$, i.e. $p<p_c$, it seems that  there exists no local solution in any reasonable sense for some initial data $h_0\in L^p(\R^d)$. In particular, in focusing case, there exists a nonnegative function $h_0\in L^p(\R^d)$ such that the \eqref{heat'} does not admit any nonnegative classical $L^p$-solution in $[0,T)$ for any $T>0$, see e.g. Brezis and Cabr\'e \cite{Bre-Cab}, Brezis and Cazenave \cite{Bre-Caz}, Haraux-Weissler\cite{Har-Weissler} and  Weissler \cite {Weiss1, Weiss2}. Also, one see book Quitnner-Souplet\cite{Qui-Soup} for many related topics and references.

In this paper, we mainly concerned with the local and global existence of solution  for some supercritical initial data $h_0\in L^p(\R^d)$ by $p<p_c$  and more generally, initial data in $ \dot H^{-\epsilon}$. For simplicity, we only consider  the Cauchy
problem for the $L^2$-critical semilinear heat equations,
\begin{equation}\label{heat}
   \left\{ \aligned
    &\partial_t h=\Delta h+\mu|h|^{\frac4d}h, \\
    &h(0,x)=h_0(x), \quad x\in \R^d,
   \endaligned
  \right.
 \end{equation}
That is, $p_c=2$\ ( i.e $\gamma=1+\frac4d $ ), we will prove that  there exists some positive constant $\varepsilon_0$ depending on $d$, such that the Cauchy problem is locally and globally wellposed for any initial data $h_0$ is radial, supported away from origin
and in the negative Sobolev space $\dot H^{-\varepsilon_0}(\R^d)$, which includes certain $L^p$-space with $p<p_c=2$ as a subspace (see Remark \ref{rem:01} below). We remark that, at present the the range of  $\epsilon_0$ in the following theorem may not be optimal to local and global existence of solution of the problem \eqref{heat}. On the other hand, we also mention that a result in Brezis and Freidman\cite{Bre-Fre} implies that the problem \eqref{heat} has no any solution ( even weak one) with a Dirac initial data $\delta$, which is in $H^{-\epsilon}(\R^d)$ for any $s>d/2$.

\begin{thm} \label{thm:main1}
Let $ \mu=\pm1$ and
 \begin{equation}\label{index}
\varepsilon_0\in\aligned
    &\Big[0, \frac{d-1}{d+2}\Big), \ \ d\ge 2.\\
   \endaligned
 \end{equation}
Suppose that $h_0\in \dot H^{-\varepsilon_0}(\R^d)$ is a radial initial data satisfying
$
\mbox{supp } h_0\subset \{x:|x|\ge1\}.
$
Then there exists a time $\delta=\delta(h_0)>0$ and a unique strong solution $$h\in C([0,\delta); L^2(\R^d)+\dot H^{-\varepsilon_0}(\R^d))\cap L^\frac{2(2+d)}{d}_{tx}([0,\delta]\times\R^d)$$ to the equation \eqref{heat} with the initial data $h_0$. Moreover, the  following two statements hold:
\begin{itemize}
\item[(1).] If $d>4$, then the solution $h$ is unique in the following sense that there exists  a unique function $w$ in $C([0,\delta], L^2(\R^d))$ such that
 \begin{equation}\label{struct}
h=e^{t\Delta}h_0+w.
 \end{equation}
\item[(2).] If  $\|h_0\|_{\dot H^{-\varepsilon_0}(\R^d)}$ is small enough, then the solution is global in time and satisfies the following decay estimate for $d\ge4$,
$$
\|h(t)\|_{L^2}\lesssim t^{-\frac{\varepsilon_0}{2}}\|h_0\|_{\dot H^{-\varepsilon_0}}, \ \ t>0.
$$
\end{itemize}
\end{thm}

\begin{remark}\label{rem:01}
 If $h_0\in L^p$ for some $p<2$, then there exists some $\epsilon_0>0$ such that $h_0\in\dot H^{-\varepsilon_0}(\R^d) $ and
$$\|h_0\|_{\dot H^{-\varepsilon_0}(\R^d)}\lesssim \|h_0\|_{L^p(\R^d)}$$
by the Sobolev embedding estimate ( see e.g. Lemma \ref{lem:radial-Sob} below ). Thus,  Theorem \ref{thm:main1} shows that the solution $h$ of the equation \eqref{heat}
exists locally for any radial and supported away from zero initial datum $h_0$  in $L^p(\R^d)$ as $p\in \Big(\frac{d^2+4d-2}{2d^2+2d},2\Big)$ and $ d\ge 2.$
 \end{remark}

\begin{remark}\label{rem:03}
It seems that the restriction $d>4$ is necessary for unconditional uniqueness. In fact, when $d=4$, the uniqueness problem is related to the ``double critical'' case ( i.e. $p=p_c=\gamma=\frac{d}{d-2}=2$ ). It was well-known that the unconditional uniqueness failed by Ni-Sacks \cite{Ni-Sacks} and Brezis and Cazenave \cite{Bre-Caz}.
 \end{remark}
Finally, it is worth mentioning  that in the defocusing case, the smallness restriction on the initial datum in the statement (2) is not necessary for global existence. Indeed, we have $h(\delta)\in L^2(\R^d)$, then it follows by considering the solution from $t=\delta$. Moreover, it is  easy to find a large class of $h_0$ satisfying the conditions of theorem above. As described in Remark \ref{rem:01},  our result  shows that the solution $h$ of the equation \eqref{heat}
exists globally on $\R+$, for any
the initial datum $h_0$ in $L^p(\R^d)$ with some $p<2$,  which is radial and supported away from zero.

The paper is organized as follows: In Section 2, we will list several useful lemmas about Littlewood-Paley theory, and space-time estimates for the solution of linear heat equation. Then in Section 3, we will give the proof of the main results, respectively.

\section{Preliminary}

\subsection{Littlewood-Paley multipliers and related inequalities}

Throughout this paper, we write $A\lesssim B$ to signify
that there exists a constant $c$ such that $A\leq cB$, while we denote $A\sim B$ when
$A\lesssim B\lesssim A$. We first define the {\it  Littlewood-Paley projection multiplier}. Let $\varphi(\xi)$ be a fixed real-valued radially symmetric bump function adapted to the ball $\{\xi\in \R^d: \ |\xi|\le 2\}$ which equals 1 on the ball $\{\xi\in \R^d: \ |\xi|\le 1\}$. Define a {\it dyadic number} to any number $N\in 2^{\mathbf{Z}}$ of the form $N=2^j$ where $j\in \mathbf{Z}$ ( the integer set). For each dyadic number $N$, we define the the Fourier multipliers
$$\widehat{P_{\le N}f}(\xi):=\varphi(\xi/N)\hat{f}(\xi), \ \ \widehat{P_{N}f}(\xi):=\varphi(\xi/N)-\varphi(2\xi/N)\hat{f}(\xi), $$
where $\hat{f}$ denotes the Fourier transform of $f$.  Moreover, define $P_{>N}=I-P_{\le N}$ and  $P_{<N}=P_{\le N}-P_N$, etc. In particular, we have the telescoping expansion:
$$P_{\le N}=\sum_{M\le N} P_M f; \ \ P_{>N}=\sum_{M> N} P_M f$$
where $M$ ranges over dyadic numbers. It was well-known that the Littlewood-Paley operators satisfy the following useful {\it Bernstein inequalities} with $s>0$ and $1\le p\le q\le\infty$ ( see e.g. Tao \cite{tao} ):
$$\|P_{\ge N}f\|_{L^p_x(\R^d)}\lesssim N^{-s}\||\nabla|^s P_{\ge N}f\|_{L_x^p(\R^d)}, \ \||\nabla|^s P_{\le N}f\|_{L^p_x(\R^d)}\lesssim N^{s}\|| P_{\le N}f\|_{L_x^p(\R^d)};$$
$$ \||\nabla|^{\pm s}  P_{\le N}f\|_{L^p_x(\R^d)}\sim N^{\pm s}\|| P_{\le N}f\|_{L_x^p(\R^d)};$$
$$\|P_{N}f\|_{L^q_x(\R^d)}\lesssim N^{(\frac{d}{p}-\frac{d}{q})}\| f\|_{L_x^p(\R^d)}, \ \|P_{\le N}f\|_{L^q_x(\R^d)}\lesssim N^{(\frac{d}{p}-\frac{d}{q})}\|f\|_{L_x^p(\R^d)};$$

Moreover, we also have the following {\it mismatch estimate}, see e.g. \cite{LiZh-APDE}.
\begin{lem}[Mismatch estimates]\label{lem:mismatch}
Let $\phi_1$ and $\phi_2$ be smooth functions obeying
$$
|\phi_j| \leq 1 \quad \mbox{ and }\quad \mbox{dist}(\emph{supp}
\phi_1,\, \emph{supp} \phi_2 ) \geq A,
$$
for some large constant $A$.  Then for $m>0$, $N\ge 1$ and $1\leq p\leq
q\leq \infty$,
\begin{align*}
\bigl\| \phi_1  P_{\le  N} (\phi_2 f)
\bigr\|_{L^q_x(\R^d)}=\bigl\| \phi_1  P_{\ge N} (\phi_2 f)
\bigr\|_{L^q_x(\R^d)}
    \lesssim_m A^{-m+\frac dq-\frac dp}N^{-m} \|\phi_2 f\|_{L^p_x(\R^d)}.
\end{align*}
\end{lem}

\subsection{Space-time estimates of linear heat equation}Let $e^{t\Delta}$ denote the heat semigroup on $\mathbb{R}^d$. Then for suitable function $f$,  $e^{t\Delta}f$ solves the linear heat equation
$$
\partial_t h=\Delta h, \ \ h(0,x)=f(x),\ t>0, \ x\in\R^d,
$$
and the solution satisfies the following fundamental space-time estimates:
\begin{lem}\label{lem:Stri} Let $f\in L^p(\R^d)$ for $1\le p\le \infty$, then
\begin{align}\label{lem:Lp-bound}
\big\|e^{t\Delta}f\big\|_{L^\infty_tL^p_x(\R^+\times\R^d)}\lesssim \|f\|_{L^p(\R^d)}.
\end{align}
Moreover, let $I\subset \R^+$, then for $f\in L^2(\R^d)$ and $F\in L^\frac{2(2+d)}{d+4}_{tx}(\R^+\times\R^d)$,
\begin{align}
\big\|\nabla e^{t\Delta}f\big\|_{L^2_{tx}(\R^+\times\R^d)}&\lesssim \|f\|_{L^2(\R^d)};\label{nabla-L2}
\end{align}
\begin{align}\label{nabla-L2'}
\big\| e^{t\Delta}f\big\|_{L^{\frac{2(2+d)}{d}}_{tx}(\R^+\times\R^d)}&\lesssim \|f\|_{L^2(\R^d)};
\end{align}
\begin{align}\label{Lpq}
\Big\|\int_0^t e^{(t-s)\Delta}F(s)ds\Big\|_{L^\infty_tL^2_x\ \cap \ L^\frac{2(2+d)}{d}_{tx}\cap\ L^2_t\dot H^1_x(I\times\R^d)}\lesssim \|F\|_{L^\frac{2(2+d)}{d+4}_{tx}(\R^+\times\R^d)}.
\end{align}
\end{lem}

We can give some remarks on the inequalities $\eqref{lem:Lp-bound}-\eqref{Lpq}$ above as follows:

 (i). The estimate \eqref{lem:Lp-bound} is classical and  immediately follows from the Younger inequality by the following heat kernel integral:
$$(e^{t\Delta}f)(x)=( 4\pi t)^{-d/2}\int_{\R^d}e^{-|x-y|^2/4t}f(y)dy, \ t>0.$$
More generally, for all $1\le p\le q\le\infty$, the following (decay) estimates hold:
\begin{align}
\|e^{t\Delta}f\|_{L^q(\R^d)}\lesssim t^{\frac{d}{2}(\frac{1}{q}-\frac{1}{p})} \|f\|_{L^p(\R^d)}, \  \ t>0.\label{Linear-decay}
\end{align}

 (ii). The estimate \eqref{nabla-L2} is equivalent to  a kind of square-function inequality on $L^2(\R^d)$, which can be reformulated as
$$\Big\| \Big( \int_0^\infty |\sqrt{t}\nabla e^{t\Delta}f|^2\frac{dt}{t}\Big)^{\frac{1}{2}}\Big\|_{L^2(\R^d)}\lesssim \|f\|_{L^2(\R^d)},$$
which follows directly by the Plancherel's theorem, and also holds in the $L^p(\R^d)$ for $1<p<\infty$ ( see e.g. Stein\cite[p. 27-46]{Stein} ).

(iii). The estimate \eqref{nabla-L2'} can be obtained by interpolation between the \eqref{lem:Lp-bound} and \eqref{nabla-L2}:
$$
\big\|e^{t\Delta}f\big\|_{L^\frac{2(2+d)}{d}_{tx}(\R^+\times\R^d)}\lesssim \big\|e^{t\Delta}f\big\|_{L^\infty_tL^2_x(\R^+\times\R^d)}^{\frac 2{d+2}}
\big\|\nabla e^{t\Delta}f\big\|_{L^2_{tx}(\R^+\times\R^d)}^{\frac d{d+2}}.
$$

(iv). The estimate \eqref{Lpq} consists of the three same type inequalities with the different norms $L^\infty_tL^2_x$, $ L^\frac{2(2+d)}{d}_{tx}$ and  $L^2_t\dot H^1_x$ on the left side. As shown in (iii) above, the second norm $ L^\frac{2(2+d)}{d}_{tx}$ can be controlled by interpolation between $L^\infty_tL^2_x$ and $L^2_t\dot H^1_x$. Because of similarity of their  proofs, we  can give a  proof to the first one, which is the special case of the following lemma.
 It is worth to noting that when $p<\infty$, the estimate is $L^2$-subcritical.
\begin{lem}\label{lem:Stri-1} Let $2\le p\le \infty$, and the pair $(p_1,r_1)$ satisfy
$$
\frac2{p_1}+\frac d{r_1}=\frac d2+2+\frac2p,\quad 1\le p_1\le 2,\quad 1<r_1\le2,
$$
then
$$
\Big\|\int_0^t e^{(t-s)\Delta}F(s)ds\Big\|_{L^p_tL^2_x(\R^+\times\R^d)}\lesssim \|F\|_{L_t^{p_1}L_x^{r_1}(\R^+\times\R^d)}.
$$
\end{lem}
\begin{proof}
By Plancherel's theorem, it is equivalent that
 \begin{align}\label{lpq'}\Big\|\int_0^t e^{-(t-s)|\xi|^2}\widehat{F}(\xi, s)ds \Big\|_{L^p_tL^2_\xi(\R^+\times\R^d)}\lesssim \|F\|_{L_t^{p_1}L_x^{r_1}(\R^+\times\R^d)}.
 \end{align}
Since  by the Young inequality of the convolution on $\R^+$, for any $1\le p_1\le p\le \infty$,
$$\Big\|\int_0^t e^{-(t-s)|\xi|^2}\widehat{F}(\xi, s)ds\Big\|_{L^p(\R^+)}\lesssim \Big\||\xi|^{-(\frac2p+\frac2{p_1'})}\widehat{F}(\xi,\cdot)\Big\|_{L_t^{p_1}(\R^+)}.$$
Note that $p_1\le 2\le p$, thus by Minkowski's inequality,  Plancherel's theorem, Sobolev's embedding we obtain
 \begin{align*}
 \Big\|\int_0^t e^{-(t-s)|\xi|^2}\widehat{F}(\xi, s)ds \Big\|_{L^p_tL^2_\xi(\R^+\times\R^d)}&\lesssim \Big\||\xi|^{-(\frac2p+\frac2{p_1'})}\widehat{F}(\xi,\cdot)\Big\|_{L^2_\xi L_t^{p_1}(\R^+\times\R^d)}\\ &\lesssim \Big\||\nabla|^{-(\frac2p+\frac2{p_1'})}F\Big\|_{L_t^{p_1}L^2_x(\R^+\times\R^d)}\lesssim\|F\|_{L_t^{p_1}L_x^{r_1}(\R^+\times\R^d)}.
\end{align*}
which gives the desired estimate \eqref{lpq'}.
\end{proof}

Finally, we also need the following maximal $L^p$-regularity result for the heat flow. See Lemarie-Rieusset's book \cite[P.64]{LR-Book} for example.

\begin{lem}\label{lem:max-Lp} Let $p\in (1,\infty),q\in (1,\infty)$, and let $T\in (0,\infty]$, then the operator $A$ defined by $$f(t,x)\mapsto \int_0^t e^{(t-s)\Delta}\Delta f(s,\cdot)\,ds$$ is bounded from $L^p((0,T),L^q(\R^d))$ to $L^p((0,T),L^q(\R^d))$.
\end{lem}
\section{Proof of Theorem \ref{thm:main1}}

In this section, we will divide several subsection to finish the proof of  Theorem \ref{thm:main1}. For the end, we first establish a supercritical estimate on the linear heat flow in the following subsection.
\subsection{A supercritical estimate on the linear heat flow} Let us recall the following radial Sobolev embedding, see \cite{TaViZh} for example.
\begin{lem}\label{lem:radial-Sob}
Let $\alpha,q,p,s$ be the parameters which satisfy
$$
\alpha>-\frac dq;\quad \frac1q\le \frac1p\le \frac1q+s;\quad 1\le p,q\le \infty; \quad 0<s<d
$$
with
$$
\alpha+s=d(\frac1p-\frac1q).
$$
Moreover, let at most one of the following equalities hold:
$$
p=1,\quad p=\infty,\quad q=1,\quad q=\infty,\quad \frac1p=\frac1q+s.
$$
Then the radial Sobolev embedding inequality holds:
\begin{align*}
\big\||x|^\alpha u\big\|_{L^q(\R^d)}\lesssim \big\||\nabla|^su\big\|_{L^p(\R^d)},
\end{align*}

\end{lem}
\begin{lem}\label{prop:superest} For any $q>2$ and any $\gamma\in \big(\frac12-\frac3q,1-\frac4q\big)$, suppose that the radial function $f\in H^\gamma(\R^d)$ satisfying
$$
\mbox{supp }f\subset \{x:|x|\ge1\},
$$
then
$$
\big\|e^{t\Delta}f\big\|_{L^q_{tx}(\R^+\times\R^d)}\lesssim \big\||\nabla|^\gamma f\big\|_{L^2_{x}(\R^d)}.
$$
\end{lem}
\begin{proof} By Lemma \ref{lem:Lp-bound}, we have
$$
\big\|e^{t\Delta}f\big\|_{L^\infty_{tx}(\R^+\times\R^d)}\lesssim \|f\|_{L^\infty(\R^d)}.
$$
Let $\alpha = \frac d2-s>0$ and $s\in (\frac12,1)$, then by Lemma \ref{lem:radial-Sob} we have
$$
\|f\|_{L^\infty(\R^d)}\lesssim\||x|^\alpha f\|_{L^\infty(\R^d)}\lesssim \big\||\nabla|^{s}f\big\|_{L^2(\R^d)},
$$
where the first inequality above has used the condition $\mbox{supp }f\subset \{x:|x|\ge1\}$.
Thus we get that
\begin{align}
\big\|e^{t\Delta}f\big\|_{L^\infty_{tx}(\R^+\times\R^d)}\lesssim \big\||\nabla|^{s}f\big\|_{L^2(\R^d)}.\label{12.44}
\end{align}
Interpolation between this last estimate and \eqref{nabla-L2}, gives our desired estimates.
\end{proof}

\subsection{Local theory and global criterion } We use $\chi_{\le a}$ for $a\in \R^+$ to denote the smooth function
\begin{align*}
\chi_{\le a}(x)=\left\{ \aligned
1, \ & |x|\le a,\\
0,    \ &|x|\ge \frac{11}{10} a,
\endaligned
  \right.
\end{align*}
and set $\chi_{\ge a}=1-\chi_{\le a}$.

Now write
\begin{align}\label{initial-decp}
h_0=v_0+w_0,
\end{align}
where
$$
v_0=\chi_{\ge \frac12}\big(P_{\ge N}h_0\big),\quad w_0=h_0-v_0.
$$
Then we will first claim that $w_0\in L^2(\R^d)$, and
\begin{align}
\|w_0\|_{L^2(\R^d)}\lesssim N^{\varepsilon_0}\big\|h_0\big\|_{\dot H^{-\varepsilon_0}(\R^d)}.\label{eq:21.43}
\end{align}
Note that $w_0=\chi_{\le \frac12}\big(P_{\ge N}h_0\big)+P_{< N}h_0$.
Firstly, we give the following estimate on the first part, which is a consequence of Lemma \ref{lem:mismatch}.
\begin{lem}\label{lem:h-mis}
Let $h_0$ be the function satisfying the hypothesis in Theorem \ref{thm:main1}, then
\begin{align}
\big\|\chi_{\le \frac12}\big(P_{\ge N}h_0\big)\big\|_{L^2(\R^d)}\lesssim
N^{-1}\big\|h_0\big\|_{\dot H^{-\varepsilon_0}(\R^d)}.
\end{align}
\end{lem}
\begin{proof}
By the support property of $h_0$, we may write
\begin{align}
\chi_{\le \frac12}&\big(P_{\ge N}h_0\big)
=\chi_{\le \frac12}\big(P_{\ge N}\chi_{\ge \frac{9}{10}}h_0\big)\notag\\
=&
\chi_{\le \frac12}\big(P_{\ge N}\chi_{\ge \frac{9}{10}}P_{\le 2N}h_0\big)
+\sum\limits_{M=4N}^\infty\chi_{\le \frac12}P_{\ge N}\big(\chi_{\ge \frac{9}{10}}P_Mh_0\big).\label{20.55}
\end{align}
By Lemma \ref{lem:mismatch} and Bernstein's inequality, we have
\begin{align}
\big\|\chi_{\le \frac12}\big(P_{\ge N}\chi_{\ge \frac{9}{10}}P_{\le 2N}h_0\big)
\big\|_{L^2(\R^d)}
&\lesssim
N^{-10}\big\|P_{\le 2N}h_0\big\|_{L^2(\R^d)}\notag\\
&\lesssim
N^{-1}\big\|h_0\big\|_{\dot  H^{-\varepsilon_0}(\R^d)}.
\label{20.55-I}
\end{align}
Moreover, since $P_{\ge N}=I-P_{< N}$ and $M>2N$, we obtain
\begin{align*}
\chi_{\le \frac12}&P_{\ge N}\big(\chi_{\ge \frac{9}{10}}P_Mh_0\big)
=-\chi_{\le \frac12}P_{< N}\big(\chi_{\ge \frac{9}{10}}P_Mh_0\big)\\
&=-\chi_{\le \frac12}P_{< N}\Big(P_{\ge\frac18M}\big(\chi_{\ge \frac{9}{10}}\big)P_Mh_0\Big),
\end{align*}
where $P_{\ge\frac18M}(\chi_{\ge \frac{9}{10}})$ denotes the high frequency truncation of the bump function $\chi_{\ge \frac{9}{10}}$.

Note that
\begin{align*}
\Big\|\chi_{\le \frac12}P_{<N}\Big(P_{\ge\frac18M}&\big(\chi_{\ge \frac{9}{10}}\big)P_Mh_0\Big)\Big\|_{L^2(\R^d)}
\lesssim \big\|P_{\ge\frac18M}\big(\chi_{\ge \frac{9}{10}}\big)\big\|_{L^\infty(\R^d)}\big\|P_Mh_0\big\|_{L^2(\R^d)}\\
\lesssim &M^{-2}\big\|\Delta P_{\ge\frac18M}\big(\chi_{\ge \frac{9}{10}}\big)\big\|_{L^\infty(\R^d)}\big\|P_Mh_0\big\|_{L^2(\R^d)}\\
\lesssim &M^{-1}\big\|h_0\big\|_{\dot H^{-\varepsilon_0}(\R^d)}.
\end{align*}
Hence, we have
\begin{align*}
\Big\|\chi_{\le \frac12}&P_{\ge N}\Big(\chi_{\ge \frac{9}{10}}P_Mh_0\Big)\Big\|_{L^2(\R^d)}
\lesssim M^{-1}\big\|h_0\big\|_{\dot H^{-\varepsilon_0}(\R^d)}.
\end{align*}
Therefore, taking summation, we obtain
\begin{align}
\sum\limits_{M=4N}^\infty\big\|\chi_{\le \frac12}P_{\ge N}\big(\chi_{\ge \frac{9}{10}}P_Mh_0\big)\big\|_{L^2(\R^d)}
\lesssim N^{-1}\big\|h_0\big\|_{\dot H^{-\varepsilon_0}(\R^d)}.
\label{20.55-II}
\end{align}
Inserting \eqref{20.55-I} and \eqref{20.55-II} into \eqref{20.55}, we prove the lemma.
\end{proof}
Moreover, by the Bernstein estimate,
$$
\big\|P_{<N}h_0\big\|_{L^2(\R^d)}\lesssim
N^{\varepsilon_0}\big\|h_0\big\|_{\dot  H^{-\varepsilon_0}(\R^d)}.
$$
Then this last estimate combining with Lemma \ref{lem:h-mis} gives \eqref{eq:21.43}.

Second, we claim that
\begin{align}
\big\|v_0\big\|_{\dot  H^{-\varepsilon_0}(\R^d)}\lesssim \big\|h_0\big\|_{\dot  H^{-\varepsilon_0}(\R^d)}.\label{11.56}
\end{align}
Indeed,
\begin{align*}
\big\|v_0\big\|_{\dot  H^{-\varepsilon_0}(\R^d)}
\lesssim &
\big\|h_0\big\|_{\dot  H^{-\varepsilon_0}(\R^d)}
+\big\|\chi_{\le \frac12}\big(P_{\ge N}h_0\big)\big\|_{\dot  H^{-\varepsilon_0}(\R^d)}.
\end{align*}
Hence, we only consider the latter term. By Sobolev's embedding and H\"{o}lder's inequality, we have
\begin{align*}
\big\|\chi_{\le \frac12}\big(P_{\ge N}h_0\big)\big\|_{\dot  H^{-\varepsilon_0}(\R^d)}
\lesssim
\big\|\chi_{\le \frac12}\big(P_{\ge N}h_0\big)\big\|_{L^2(\R^d)}.
\end{align*}
Hence \eqref{11.56} follows from Lemma \ref{lem:h-mis}.

We denote
$$
v_L(t)=e^{t\Delta}v_0.
$$
Then $v_L$ is globally existence, and by Plancherel's theorem and \eqref{11.56}
\begin{align}\label{linear part}
\big\|v_L(t)\big\|_{L^\infty_t \dot H^{-\varepsilon_0}_x(\R^+\times\R^d)}\lesssim \big\|v_0\big\|_{\dot  H^{-\varepsilon_0}(\R^d)}\lesssim \big\|h_0\big\|_{\dot  H^{-\varepsilon_0}(\R^d)},
\end{align}
Moreover, let $\epsilon$ be a sufficiently small positive constant, then we claim that
\begin{align}
\big\|v_L(t)\big\|_{L^\frac{2(2+d)}{d}_{tx}(\R^+\times\R^d)}\lesssim N^{-\frac{d-1}{d+2}+\varepsilon_0+\epsilon}\big\|h_0\big\|_{\dot  H^{-\varepsilon_0}(\R^d)}.\label{1.14}
\end{align}
Indeed, let $\gamma=-\frac{d-1}{d+2}+\epsilon$, then by Lemma \ref{prop:superest},
\begin{align*}
\big\|v_L(t)\big\|_{L^\frac{2(2+d)}{d}_{tx}(\R^+\times\R^d)}\lesssim \big\||\nabla|^{\gamma}\chi_{\ge \frac12}\big(P_{\ge N}h_0\big)\big\|_{L^2(\R^d)}.
\end{align*}
Note that
$$
\big\||\nabla|^{\gamma}\chi_{\ge \frac12}\big(P_{\ge N}h_0\big)\big\|_{L^2(\R^d)}
\le \big\||\nabla|^{\gamma}\big(P_{\ge N}h_0\big)\big\|_{L^2(\R^d)}+\big\||\nabla|^{\gamma}\chi_{\le \frac12}\big(P_{\ge N}h_0\big)\big\|_{L^2(\R^d)}.
$$
For the former term, since $\gamma<-\varepsilon_0$, by Bernstein's inequality,
$$
\big\||\nabla|^{\gamma}\big(P_{\ge N}h_0\big)\big\|_{L^2(\R^d)}
\lesssim
N^{\gamma+\varepsilon_0}\big\|h_0\big\|_{\dot H^{-\varepsilon_0}(\R^d)}.
$$
So we only need to estimate the latter term. Let $q$ be the parameter satisfying
$$
\frac1q=\frac12-\frac\gamma d,
$$
then $q>1$.
Since $\gamma<0$, by Sobolev's and H\"older's inequalities,
$$
\big\||\nabla|^{\gamma}\chi_{\le \frac12}\big(P_{\ge N}h_0\big)\big\|_{L^2(\R^d)}
\lesssim \big\|\chi_{\le \frac12}\big(P_{\ge N}h_0\big)\big\|_{L^q(\R^d)}
\lesssim  \big\|\chi_{\le \frac12}\big(P_{\ge N}h_0\big)\big\|_{L^2(\R^d)}.
$$
Furthermore, by Lemma \ref{lem:h-mis},
$$
 \big\|\chi_{\le \frac12}\big(P_{\ge N}h_0\big)\big\|_{L^2(\R^d)}
 \lesssim N^{-1}\big\|h_0\big\|_{\dot H^{-\varepsilon_0}(\R^d)}.
$$
Combining the last two estimates above, we obtain
$$
\big\||\nabla|^{\gamma}\chi_{\le \frac12}\big(P_{\ge N}h_0\big)\big\|_{L^2(\R^d)}
\lesssim N^{-1}\big\|h_0\big\|_{\dot H^{-\varepsilon_0}(\R^d)}.
$$
This gives \eqref{1.14}.

Now we denote $w=h-v_L$, then $w$ is the solution of the following equation,
\begin{equation}\label{heat-w}
   \left\{ \aligned
    &\partial_t w=\Delta w\pm|h|^{\frac4d}h, \\
    &w(0,x)=w_0(x)=h_0-v_0.
   \endaligned
  \right.
 \end{equation}
The following lemma is the local well-posedness and global criterion of the Cauchy problem \eqref{heat-w}.
\begin{lem}\label{lem:local} There exists $\delta>0$, such that for any $h_0$ satisfying the hypothesis in Theorem \ref{thm:main1} and $w_0=h_0-v_0$,  the Cauchy problem \eqref{heat-w} is well-posed on the time interval $[0,\delta]$, and
the solution $$w\in C_tL^2_x([0,\delta]\times\R^d)\cap L^\frac{2(2+d)}{d}_{tx}([0,\delta]\times\R^d)\cap L^2_t\dot H^1_x([0,\delta]\times\R^d).$$ Furthermore, let $T^*$ be the maximal lifespan, and  suppose that
$$w\in L^\frac{2(2+d)}{d}_{tx}([0,T^*)\times\R^d),$$
then $T^*=+\infty$. In particular, if $\|h_0\|_{\dot  H^{-\varepsilon_0}(\R^d)}\ll 1$, then $T^*=+\infty$.
\end{lem}
\begin{proof}
For local well-posedness, we only show that the solution $w\in L^\infty_tL^2_x([0,\delta]\times\R^d)\cap L^\frac{2(2+d)}{d}_{tx}([0,\delta]\times\R^d)\cap L^2_t\dot H^1_x([0,\delta]\times\R^d)$ for some $\delta>0$. Indeed, the local well-posedness with the lifespan $[0, \delta)$ is then followed by the standard fixed point argument.
By Duhamel's formula, we have
$$
w(t)=e^{t\Delta}w_0\pm\int_0^te^{(t-s)\Delta}|h(s)|^\frac4dh(s)\,ds.
$$
Then by Lemma \ref{lem:Stri}, for any $t_*\le \delta$,
\begin{align*}
\big\|w\big\|_{L^\frac{2(2+d)}{d}_{tx}([0,t_*]\times\R^d)}
\lesssim &\|e^{t\Delta}w_0\|_{L^\frac{2(2+d)}{d}_{tx}([0,t_*]\times\R^d)}+\big\||h|^\frac4dh\big\|_{L^\frac{2(2+d)}{d+4}_{tx}([0,t_*]\times\R^d)}\\
\lesssim &\|e^{t\Delta}w_0\|_{L^\frac{2(2+d)}{d}_{tx}([0,\delta]\times\R^d)}+\big\|h\big\|_{L^\frac{2(2+d)}{d}_{tx}([0,t_*]\times\R^d)}^{\frac4d+1}.
\end{align*}
Note that
$$
\big\|h\big\|_{L^\frac{2(2+d)}{d}_{tx}([0,t_*]\times\R^d)}\lesssim \big\|v_L\big\|_{L^\frac{2(2+d)}{d}_{tx}(\R^+\times\R^d)}
+\big\|w\big\|_{L^\frac{2(2+d)}{d}_{tx}([0,t_*]\times\R^d)},
$$
let $\eta_0=(\frac4d+1)\big(\frac{d-1}{d+2}-\varepsilon_0-\epsilon\big)>0$, then using \eqref{1.14}, we obtain
\begin{align*}
\big\|w\big\|_{L^\frac{2(2+d)}{d}_{tx}([0,t_*]\times\R^d)}
\lesssim &\|e^{t\Delta}w_0\|_{L^\frac{2(2+d)}{d}_{tx}([0,\delta]\times\R^d)}+N^{-\eta_0}\big\|h_0\big\|_{\dot  H^{-\varepsilon_0}(\R^d)}^{\frac4d+1}
+\big\|w\big\|_{L^\frac{2(2+d)}{d}_{tx}([0,t_*]\times\R^d)}^{\frac4d+1}.
\end{align*}
Noting that either $\|h_0\|_{\dot  H^{-\varepsilon_0}(\R^d)}\ll 1$, or choosing $\delta$ small enough and $N$ large enough, we have
$$
\|e^{t\Delta}w_0\|_{L^\frac{2(2+d)}{d}_{tx}([0,\delta]\times\R^d)}+N^{-\eta_0}\big\|h_0\big\|_{\dot  H^{-\varepsilon_0}(\R^d)}^{\frac4d+1}\ll1,
$$
then by the continuity argument, we
\begin{align}
\big\|w\big\|_{L^\frac{2(2+d)}{d}_{tx}([0,\delta]\times\R^d)}\lesssim \|e^{t\Delta}w_0\|_{L^\frac{2(2+d)}{d}_{tx}([0,\delta]\times\R^d)}+N^{-\eta_0}\big\|h_0\big\|_{\dot  H^{-\varepsilon_0}(\R^d)}^{\frac4d+1}. \label{11.47}
\end{align}
Further, by Lemma \ref{lem:Stri} again,
\begin{align*}
\big\|w\big\|_{L^2_t\dot H^1_x([0,\delta]\times\R^d)}+\sup\limits_{t\in [0,\delta]}\big\|w\big\|_{L^2_x(\R^d)}
\lesssim &\|w_0\|_{L^2_x(\R^d)}+\big\||h|^\frac4dh\big\|_{L^\frac{2(2+d)}{d+4}_{tx}([0,\delta]\times\R^d)}\\
\lesssim &\|w_0\|_{L^2_x(\R^d)}+\big\|v_L\big\|_{L^\frac{2(2+d)}{d}_{tx}([0,\delta]\times\R^d)}^{\frac4d+1}
+\big\|w\big\|_{L^\frac{2(2+d)}{d}_{tx}([0,\delta]\times\R^d)}^{\frac4d+1}.
\end{align*}
Hence, using \eqref{1.14} and \eqref{11.47}, we obtain
\begin{align*}
\big\|w\big\|_{L^2_t\dot H^1_x([0,\delta]\times\R^d)}+\sup\limits_{t\in [0,\delta]}\big\|w\big\|_{L^2_x(\R^d)}
\le C,
\end{align*}
for some $C=C(N,\big\|h_0\big\|_{\dot  H^{-\varepsilon_0}(\R^d)})>0$.

Suppose that
$$w\in L^\frac{2(2+d)}{d}_{tx}([0,T^*)\times\R^d),$$
then if $T^*<+\infty$, we have
\begin{align*}
\big\|w(T^*)\big\|_{L^2_x(\R^d)}
\lesssim &\|e^{t\Delta}w_0\|_{L^\frac{2(2+d)}{d}_{tx}([0,T^*]\times\R^d)}+\big\||h|^\frac4dh\big\|_{L^\frac{2(2+d)}{d+4}_{tx}([0,T^*)\times\R^d)}\\
\lesssim &\|w_0\|_{L^2_{x}(\R^d)}+N^{-\eta_0}\big\|h_0\big\|_{\dot  H^{-\varepsilon_0}(\R^d)}^{\frac4d+1}
+\big\|w\big\|_{L^\frac{2(2+d)}{d}_{tx}([0,T^*)\times\R^d)}^{\frac4d+1}.
\end{align*}
Hence, $w$ exists on $[0,T^*]$, and $w(T^*)\in L^2(\R^d)$. Hence, using the local theory obtained before from time $T^*$, the lifespan can be extended to $T^*+\delta$, this is contradicted with the definition of the  maximal lifespan $T^*$. Hence, $T^*=+\infty$.
\end{proof}

\subsection{Uniqueness}
Here we adopt the argument in \cite{Monn}, where the main tool is the the maximal $L^p$-regularity of the heat flow.
Let $h_1,h_2$ be two distinct solutions of \eqref{heat} with the same initial data $h_0$, and write
$$
h_1=e^{s\Delta}h_0+w_1;\quad h_2=e^{s\Delta}h_0+w_2.
$$
By the Duhamel formula, we have
\begin{align*}
w_1(t)=&\int_0^t e^{(t-s)\Delta}|e^{s\Delta}h_0+w_1|^{\frac4d}\big(e^{s\Delta}h_0+w_1\big)\,ds;\\
w_2(t)=&\int_0^t e^{(t-s)\Delta}|e^{s\Delta}h_0+w_2|^{\frac4d}\big(e^{s\Delta}h_0+w_2\big)\,ds.
\end{align*}
Denote $w=w_1-w_2$, then $w$ obeys
\begin{align*}
w(t)=&\int_0^t e^{(t-s)\Delta}\Big[|e^{s\Delta}h_0+w_1|^{\frac4d}\big(e^{s\Delta}h_0+w_1\big)-|e^{s\Delta}h_0+w_2|^{\frac4d}\big(e^{s\Delta}h_0+w_2\big)\Big]\,ds.
\end{align*}
Note that there exists an absolute constant $C>0$ such that
\begin{align*}
\Big||e^{s\Delta}h_0+w_1|^{\frac4d}\big(e^{s\Delta}h_0&+w_1\big)-|e^{s\Delta}h_0+w_2|^{\frac4d}\big(e^{s\Delta}h_0+w_2\big)\Big|\\
&\le C\Big(|e^{s\Delta}h_0|^{\frac4d}+|w_1|^{\frac4d}+|w_2|^{\frac4d}\Big)|w|.
\end{align*}
Then by the positivity of the heat kernel, we have
\begin{align*}
|w(t)|\le C\int_0^t e^{(t-s)\Delta}\Big(|e^{s\Delta}h_0|^{\frac4d}+|w_1(s)|^{\frac4d}+|w_2(s)|^{\frac4d}\Big)|w(s)|\,ds.
\end{align*}
Then we get that for $2\le p<\infty$, $\tau\in (0,\delta]$,
\begin{align*}
\|w\|_{L^p_t((0,\tau); L^2(\R^d))}
\lesssim &\Big\|\int_0^t e^{(t-s)\Delta}|e^{s\Delta}h_0|^{\frac4d}|w(s)|\,ds\Big\|_{L^p_t((0,\tau); L^2(\R^d))}\\
&\quad +\Big\|\int_0^t e^{(t-s)\Delta}\Big(|w_1(s)|^{\frac4d}+|w_2(s)|^{\frac4d}\Big)|w(s)|\,ds\Big\|_{L^p_t((0,\tau); L^2(\R^d))}.
\end{align*}
For the first term in the right-hand side above, using Lemma \ref{lem:Stri-1} and choosing $p$ large enough, we have
\begin{align*}
\Big\|\int_0^t e^{(t-s)\Delta}|e^{s\Delta}h_0|^{\frac4d}|w(s)|\,ds\Big\|_{L^p_t((0,\tau); L^2(\R^d))}
\lesssim \Big\||e^{s\Delta}h_0|^{\frac4d}|w(s)|\Big\|_{L^{p_1}_t((0,\tau); L^{r_1}(\R^d))},
\end{align*}
where we have chose $(p_1,r_1)$ that
$$
\frac1{p_1}=\frac2{d+2}+\frac1p;\quad \frac1{r_1}=\frac2{d+2}+\frac12.
$$
(Note that $d>4$ and $p$ is large, we have that $p_1\in (1,2), r_1\in (1,2)$). Hence, by H\"older's inequality, we obtain that
\begin{align*}
\Big\|\int_0^t e^{(t-s)\Delta}|e^{s\Delta}&h_0|^{\frac4d}|w(s)|\,ds\Big\|_{L^p_t((0,\tau); L^2(\R^d))}\\
\lesssim &\Big\||e^{s\Delta}h_0|^{\frac4d}|w(s)|\Big\|_{L^{p_1}_t((0,\tau); L^{r_1}(\R^d))}\\
\lesssim & \big\|e^{s\Delta}h_0\big\|_{L^{\frac{2(d+2)}{d}}_{tx}\big((0,\tau)\times \R^d\big)}^{\frac4d}\big\|w\big\|_{L^{p}_t((0,\tau); L^2(\R^d))}.
\end{align*}
For the second term in the right-hand side above, using Lemma \ref{lem:max-Lp},
\begin{align*}
\Big\|\int_0^t e^{(t-s)\Delta}\Big(|w_1(s)&|^{\frac4d}+|w_2(s)|^{\frac4d}\Big)|w(s)|\,ds\Big\|_{L^p_t((0,\tau); L^2(\R^d))}\\
\lesssim & \Big\|(-\Delta)^{-1}\Big(\big(|w_1(s)|^{\frac4d}+|w_2(s)|^{\frac4d}\big)|w(s)|\Big)\Big\|_{L^{p}_t((0,\tau); L^2(\R^d))}.
\end{align*}
Since $d>4$, by Sobolev's embedding, we further have
\begin{align*}
\Big\|\int_0^t e^{(t-s)\Delta}\Big(|w_1(s)&|^{\frac4d}+|w_2(s)|^{\frac4d}\Big)|w(s)|\,ds\Big\|_{L^p_t((0,\tau); L^2(\R^d))}\\
\lesssim & \big\|\big(|w_1(s)|^{\frac4d}+|w_2(s)|^{\frac4d}\big)|w(s)|\big\|_{L^{p}_t((0,\tau); L^\frac{2d}{d+4}(\R^d))}\\
\lesssim & \big(\|w_1\|_{L^\infty_t((0,\tau); L^2(\R^d))}^{\frac4d}+\|w_1\|_{L^\infty_t((0,\tau); L^2(\R^d))}^{\frac4d}\big)\|w\|_{L^{p}_t((0,\tau); L^2(\R^d))}.
\end{align*}
Collection the estimates above, we obtain that
\begin{align}
\|w\|_{L^{p}_t((0,\tau); L^2(\R^d))}
\lesssim \rho(\tau)\cdot \|w\|_{L^{p}_t((0,\tau); L^2(\R^d))}, \label{22.38}
\end{align}
where
$$
\rho(\tau)=\big\|e^{s\Delta}h_0\big\|_{L^{\frac{2(d+2)}{d}}_{tx}\big((0,\tau)\times \R^d\big)}^{\frac4d}+\|w_1\|_{L^\infty_t((0,\tau); L^2(\R^d))}^{\frac4d}+\|w_2\|_{L^\infty_t((0,\tau); L^2(\R^d))}^{\frac4d}.
$$
By \eqref{1.14} and Lemma \ref{lem:Stri}, we have
$$
\big\|e^{s\Delta}h_0\big\|_{L^{\frac{2(d+2)}{d}}_{tx}\big((0,\tau)\times \R^d\big)} \to 0, \quad \mbox{when } \tau\to 0.
$$
Further, since $w_1,w_2\in C([0,\delta], L^2(\R^d))$, we get
$$
\lim\limits_{\tau\to0} \rho(\tau)\to 0.
$$
Hence, choosing $\tau$ small enough and from \eqref{22.38}, we obtain that $w\equiv 0$ on $t\in [0,\tau)$. By iteration, we have $w_1\equiv w_2$ on $[0,\delta]$.
This proves the first statement (1) in Theorem \ref{thm:main1}.

\subsection{$L^2$-estimates}
In this subsection, we prove the second statement (2) in Theorem \ref{thm:main1}.

Firstly,  by Lemma \ref{lem:local}, when $\|h_0\|_{\dot  H^{-\varepsilon_0}(\R^d)}\ll 1$, we immediately have the global existence of the solution for the both cases $\mu=\pm 1$.
However, in the defocusing case ($\mu=1$). the smallness of $\|h_0\|_{\dot  H^{-\varepsilon_0}(\R^d)}\ll 1$ can be cancelled.  In fact, note that $h=v_L+w$ and
\begin{align}
\big\|v_L\big\|_{L^2(\R^d)}=&\big\|e^{-t|\xi|^2}\widehat{v_0}(\xi)\big\|_{L^2_\xi(\R^d)}\notag\\
\lesssim& \big\|e^{-t|\xi|^2}|\xi|^{\varepsilon_0}\big\|_{L^\infty_\xi(\R^d)}\|v_0\|_{\dot H^{-\varepsilon_0}}
\lesssim t^{-\frac{\varepsilon_0}{2}}\|h_0\|_{\dot H^{-\varepsilon_0}}.\label{9.26}
\end{align}
Hence, from Lemma \ref{lem:local}, we have $h(\delta)\in L^2(\R^d)$. Let $I=[0,T^*)$ be the maximal lifespan of the solution $h$ of the Cauchy problem \eqref{heat}. Then from the $L^2$ estimate of the solution (by inner producing with $h$ in \eqref{heat}), we have
$$
\sup\limits_{t\in I}\|h\|_{L^2}^2+\|\nabla h\|_{L^2_{tx}(I\times \R^d)}^2\le \|h_0\|_{L^2}^2.
$$
This gives the uniform boundedness of $\big\|h\big\|_{L^\frac{2(2+d)}{d}_{tx}(I\times\R^d)}$ and thus $\big\|w\big\|_{L^\frac{2(2+d)}{d}_{tx}(I\times\R^d)}$. Then by the global criteria given in  Lemma \ref{lem:local}, we have $T^*=+\infty$.

 Secondly,  we consider the time estimate of the solution ($\mu=\pm1$). When $t\le 1$, it follows from \eqref{9.26} and Lemma \ref{lem:local}, that
$$\|h(t)\|_{L^2}\lesssim t^{-\frac{\varepsilon_0}{2}}\|h_0\|_{\dot H^{-\varepsilon_0}}, \quad  \mbox{ for any } t\in (0,1].
$$
So it remains to show the decay  estimate when $t>1$. By Duhamel's formula,  we have
\begin{align*}
\|h(t)\|_{L^2(\R^d)}\le \big\|e^{t\Delta}h_0\big\|_{L^2(\R^d)}+\Big\|\int_0^t e^{(t-s)\Delta}|h(s)|^\frac4dh(s)\,ds\Big\|_{L^2(\R^d)}.
\end{align*}
Similar as \eqref{9.26}, we have
$$
\big\|e^{t\Delta}h_0\big\|_{L^2(\R^d)}\lesssim t^{-\frac{\varepsilon_0}{2}}\|h_0\|_{\dot H^{-\varepsilon_0}}.
$$
Then using the estimate above and Lemma \ref{Linear-decay}, we further have
\begin{align*}
\|h(t)\|_{L^2(\R^d)}\lesssim & t^{-\frac{\varepsilon_0}{2}}\|h_0\|_{\dot H^{-\varepsilon_0}}+\int_0^t \Big\|e^{(t-s)\Delta}|h(s)|^\frac4dh(s)\Big\|_{L^2(\R^d)}\,ds\\
\lesssim &  t^{-\frac{\varepsilon_0}{2}}\|h_0\|_{\dot H^{-\varepsilon_0}}+\int_0^t |t-s|^{-1}\Big\||h(s)|^\frac4dh(s)\Big\|_{L^\frac{2d}{d+4}(\R^d)}\,ds\\
\lesssim &  t^{-\frac{\varepsilon_0}{2}}\|h_0\|_{\dot H^{-\varepsilon_0}}+\int_0^t |t-s|^{-1}\|h\|_{L^2(\R^d)}^{\frac4d+1}\,ds.
\end{align*}
In the last step we have used the fact $d\ge 4$ such that $\frac{2d}{d+4}\ge 1$.

Now we denote
$$
\|h\|_{X(T)}=\sup\limits_{t\in[0,T]}\Big(t^\frac{\varepsilon_0}{2}\|h(t)\|_{L^2(\R^d)}\Big).
$$
Fixing $T>1$, then for any $t\in (1,T]$,
\begin{align*}
\|h(t)\|_{L^2(\R^d)}\lesssim & t^{-\frac{\varepsilon_0}{2}}\|h_0\|_{\dot H^{-\varepsilon_0}}+\int_0^t |t-s|^{-1}s^{-\frac{\varepsilon_0}{2}(\frac4d+1)}\,ds\|h(t)\|_{X(T)}^{\frac4d+1}\\
\lesssim & t^{-\frac{\varepsilon_0}{2}}\|h_0\|_{\dot H^{-\varepsilon_0}}+t^{-\frac{\varepsilon_0}{2}(\frac4d+1)+}\|h(t)\|_{X(T)}^{\frac4d+1}\\
\lesssim & t^{-\frac{\varepsilon_0}{2}}\Big(\|h_0\|_{\dot H^{-\varepsilon_0}}+\|h(t)\|_{X(T)}^{\frac4d+1}\Big).
\end{align*}
Thus we obtain that
$$
\|h(t)\|_{X(T)}\lesssim  \|h_0\|_{\dot H^{-\varepsilon_0}}+\|h(t)\|_{X(T)}^{\frac4d+1}.
$$
By the continuity argument, we get
$$
\|h(t)\|_{X(T)}\lesssim  \|h_0\|_{\dot H^{-\varepsilon_0}}.
$$
Since the estimate is independent on $T$, we give that
$$\|h(t)\|_{L^2}\lesssim t^{-\frac{\varepsilon_0}{2}}\|h_0\|_{\dot H^{-\varepsilon_0}}, \quad  \mbox{ for any } t>1.
$$
Therefore, we obtain that
$$\|h(t)\|_{L^2}\lesssim t^{-\frac{\varepsilon_0}{2}}\|h_0\|_{\dot H^{-\varepsilon_0}}, \quad  \mbox{ for any } t>0.
$$
This proves the second statement (2) in Theorem \ref{thm:main1}.

\end{document}